\documentclass[12pt,reqno]{amsart}
\usepackage{amsfonts}
\usepackage{amsmath,amscd}
\usepackage{fullpage}
\usepackage{amssymb}
\usepackage{centernot} 
\usepackage{enumerate} 
\usepackage{parskip}
\usepackage{pb-diagram} 
\usepackage{mathrsfs}
\usepackage[OT2,T1]{fontenc}
\usepackage{seqsplit}
\usepackage{enumitem}
\usepackage{color}
\usepackage{array}
\usepackage{verbatim}
\usepackage{url}

\DeclareSymbolFont{cyrletters}{OT2}{wncyr}{m}{n}
\DeclareMathSymbol{\Sha}{\mathalpha}{cyrletters}{"58}

\definecolor{refkey}{rgb}{1,1,1}
\definecolor{labelkey}{rgb}{1,1,1}
\definecolor{cite}{rgb}{0.9451,0.2706,0.4941}
\definecolor{ruri}{rgb}{0.0078,0.4022,0.8010}

\usepackage[%
bookmarks=true,bookmarksnumbered=true,%
colorlinks=true,linkcolor=ruri,citecolor=red%
 ]{hyperref}

\makeindex 

\def\F{{\rm \mathbb{F}}}
\def\Z{{\rm \mathbb{Z}}}

\def\Q{{\rm \mathbb{Q}}}

\def\C{{\rm \mathbb{C}}}

\def\R{{\rm \mathbb{R}}}

\def\P{{\rm \mathbb{P}}}

\def\p{{\rm \mathfrak{p}}}

\def\O{{\rm \mathcal{O}}}

\def\A{{\rm \mathbb{A}}}

\def\Nm{{\rm Nm}}

\def\avg{{\rm avg}}

\def\Cl{{\rm Cl}}

\def\Disc{{\rm Disc}}

\def\SL{{\rm SL}}

\def\Sym{{\rm Sym}}

\def\GL{{\rm GL}}
\def\Gal{{\rm Gal}}

\def\Sel{{\rm Sel}}

\def\Spec{{\rm Spec}}

\newcommand{\minus}{\scalebox{0.70}[1.0]{$-$}}

\newcommand{\eps}{\varepsilon}
\newcommand{\rank}{\mathop{\mathrm{rk}}}

\newcommand{\Prob}{\mathrm{Prob}}

\numberwithin{equation}{section}

\newtheorem{theorem}{Theorem}
\newtheorem{lemma}[theorem]{Lemma}

\newtheorem{corollary}[theorem]{Corollary}
\newtheorem{proposition}[theorem]{Proposition}

\usepackage{marginnote,color}
\usepackage[usenames,dvipsnames]{xcolor}

\usepackage{tikz}

\def\shownotes{\def\inline##1##2##3{ \begin{adjustwidth}{3mm}{7mm}\mbox{}\par \noindent
{\color{##1}\hspace{-1.9cm}{\large ##2}\vspace{-\baselineskip}\\##3}
\newline\end{adjustwidth}} \def\inlinewide##1##2##3{ \begin{adjustwidth}{0mm}{0cm}\mbox{}\par \noindent
{\color{##1}\hspace{-1.6cm}{\large ##2}\vspace{-\baselineskip}\\##3}
\newline\end{adjustwidth}}  \def\marg##1##2##3{\marginnote{\color{##1}{\large ##2}\\{\small ##3}}[-.8cm]}}

\shownotes

\makeatletter
\let\@@pmod\pmod
\DeclareRobustCommand{\pmod}{\@ifstar\@pmods\@@pmod}
\def\@pmods#1{\mkern4mu({\operator@font mod}\mkern 6mu#1)}
\makeatother

\begin{document}
\setlength{\parskip}{2pt} 
\setlength{\parindent}{22.5pt}
\title{A positive proportion of quartic fields are not monogenic yet have no local obstruction to being~so}



\author{Levent Alp\"oge}
\author{Manjul Bhargava}
\author{Ari Shnidman}

\begin{abstract}
We show that a positive proportion of quartic fields are not monogenic, despite having no local obstruction to being monogenic. Our proof builds on the corresponding result for cubic fields that we obtained in a previous work.  Along the way, we also prove that a positive proportion of quartic rings of integers do not arise as the invariant order of an integral binary quartic form despite having no local obstruction.  
\end{abstract}

\maketitle

\vspace{-.25in}

\section{Introduction}
In a previous paper~\cite{ABS}, we proved that a positive proportion of cubic fields are not monogenic despite having no local obstruction to being monogenic.  
The  purpose of this paper is to prove the analogous result for quartic fields. 

\begin{theorem}\label{S4main}
Let $r = 0,\, 2,$ or $4$. When isomorphism classes of quartic fields with $r$ real embeddings are ordered by absolute discriminant, a positive proportion are not monogenic and yet have no local obstruction to being monogenic. 
\end{theorem}

\noindent In other words, we prove that a positive proportion of quartic fields, up to isomorphism, are non-monogenic for truly global reasons. Moreover, we exhibit such  positive proportions of non-monogenic quartic fields having each of the three possible real signatures.

To state precisely what it means  for a number field of degree $n$ to have ``no local obstruction to being monogenic'',  recall  the {\it index
form} $f_K:\O_K/\Z \to \wedge^{n}\O_K$ of $K$  defined by $\alpha\mapsto 1\wedge\alpha\wedge\alpha^2\wedge\cdots\wedge\alpha^{n-1}$, which may be viewed as a homogeneous form of degree~$n \choose 2$ in $n-1$ variables.
A number field $K$ is monogenic if and only if $f_K$ represents~$\pm1$ over $\Z$.  We~say that $K$ {\it has no local obstruction to being monogenic}
if $f_K$ represents~$1$ over $\Z_p$ for all primes~$p$, or represents~$-1$ over $\Z_p$ for all primes~$p$.
We note that the index form of a number field~$K$ always represents~$+1$ and $-1$ over $\R$ (see~Lemma~\ref{lem:infinity}), and hence it is not necessary to consider obstructions to monogenicity over $\R$.

Our proof of Theorem \ref{S4main} actually proves something stronger.  Recall that a homogeneous polynomial $F \in \Z[x,y]$ of degree $n$ determines a based ring $R_F$ of rank $n$ --- called the {\it invariant order of $F$} --- that satisfies  $\Disc(R_F) = \Disc(F)$.  If $F(x,y) =\sum_{i = 0}^n a_i x^{n-i}y^i$ is irreducible of degree $n$ and $\theta$ is a root of $F(x,1)$, then a $\Z$-basis for $R_F$ is $\langle 1, a_0\theta, a_0\theta^2 + a_1\theta, \ldots, \sum_{i = 0}^{n-2} a_i\theta^{n-1 - i}\rangle$; see~\cite{BM,Nakagawa}.
In coordinate-free terms, $R_F$ is the ring of global  functions on the $\Z$-scheme $\mathrm{Proj}\, \Z[x,y]/(F)$;
see \cite{Wood-binary}. 

 We say that a ring $R$ is {\it binary} if $R$ is isomorphic to the invariant order of some homogeneous form $F(x,y) \in \Z[x,y]$.  A number field $K$ is {\it binary} if $\O_K$ is.  If $K$ is monogenic, i.e., $\O_K\simeq \Z[x]/(f(x))$ for a monic integer polynomial $f$, then $K$ is also binary, as it is  isomorphic to the invariant order of the homogenization $F(x,y)$ of $f$.  However, there are many binary number fields that are not monogenic.  In fact, every ring of rank $n \leq 3$ is binary, despite the fact that many cubic fields are not monogenic~\cite{ABS}.  On the other hand, for $n \geq 4$, it is natural to conjecture that, when ordered by discriminant, $100\%$ of number fields of degree $n$ are not binary, though it was not previously known that even a positive proportion of degree~$n$ number fields are not binary.   
 
 Our second main theorem states that a {\it positive proportion} of quartic fields  are not binary, despite having no local obstruction to being binary; thus they are not binary for truly global reasons. In fact, we show that this holds even if we restrict to  quartic fields having any fixed real signature:  

\begin{theorem}\label{thm:notbinary}
Let $r = 0,\, 2,$ or $4$. When isomorphism classes of quartic fields with $r$ real embeddings are ordered by absolute discriminant, a positive proportion are not binary and yet have no local obstruction to being
binary. 
\end{theorem}

Just as for monogenicity, the condition that a rank $n$ ring $R$ has {\it no local obstruction to being binary} is stronger than simply demanding that $R \otimes_\Z \Z_p$ is binary over $\Z_p$ for all~$p$. To state the condition, we use the notion of an orientation of a ring $R$ of rank~$n$ over a principal ideal domain~$S$: an {\it orientation} of $R$ is a generator $\omega$ of the $S$-module~$\wedge^n R$. An {\it oriented ring of rank~$n$} over $S$ is then a pair $(R,\omega)$, where $R$ is a ring of rank $n$ over $S$ and $\omega$ is an orientation of $R$.  If $F(x,y) = \sum_{i=0}^n a_i x^{n-i}y^i\in S[x,y]$ is a binary form, then the ring $R_F$ naturally  carries the orientation $\omega_F = 1 \wedge a_0 \theta \wedge (a_0\theta^2+a_1\theta) \wedge \cdots \wedge (a_0\theta^{n-1}+\cdots+ a_{n-2}\theta)$.  An oriented ring $(R,\omega)$ of rank~$n$ over $S$ is {\it binary} if there is an isomorphism $R \simeq R_F$ such that $\omega\mapsto \omega_F$ under this isomorphism.    

If $(R,\omega)$ is an oriented ring of rank $n$, then  $(R,\omega)$ {\it has no local obstruction to being binary} if the oriented ring $(R \otimes_\Z \Z_p, \omega \otimes 1)$ of rank $n$ over~$\Z_p$ is binary for all primes $p$.\footnote{We view $\omega \otimes 1 \in (\wedge^n R) \otimes_\Z \Z_p$ as an element of $\wedge^n(R \otimes_\Z \Z_p)$ via the canonical isomorphism.}  If $R$ is an (unoriented) ring of rank~$n$, we say that {\it $R$ has no local obstruction to being binary} if there exists an orientation $\omega$ such that $(R,\omega)$ has no local obstruction to being binary.  Since, for a ring of rank $n$ over~$\Z$, there are exactly two possible choices of orientation, this definition naturally extends the earlier-stated notion for a ring of rank $n$ to have no local obstruction to being monogenic. 

\subsection*{Geometric formulation}
Theorems \ref{S4main} and \ref{thm:notbinary} can be phrased as failures of local-to-global principles for certain embedding problems over $\Spec \, \Z$. Indeed, 
a number field $K$ is monogenic if and only if there is a closed immersion $\Spec \, \O_K \hookrightarrow \A^1_\Z = \Spec\, \Z[x]$, as giving such a closed immersion is equivalent to giving a surjection $\Z[x] \to \O_K$.
Similarly, a number field~$K$ is
binary if and only if there is a closed immersion $\Spec\, \O_K \hookrightarrow \P^1_\Z = \mathrm{Proj}\, \Z[x,y]$ (see~\cite[Prop.\ 2.9]{Wood-binary}). 
If $\omega \in \wedge^n_\Z\O_K$, then a closed immersion $\Spec \, \O_K \hookrightarrow\A^1_\Z$ is {\it $\omega$-oriented} if $\omega = \bigwedge_{i = 0}^{n-1} x^i$.
Similarly, $\Spec \, \O_K \hookrightarrow \P^1_\Z$ is {\it $\omega$-oriented} if $\omega = a_0^{n-1}\bigwedge_{i = 0}^{n-1} (x/y)^i$.  
 
In this language, the Levi--Delone--Faddeev  correspondence implies that every cubic field~$K$ admits exactly one closed immersion $\Spec \, \O_K \hookrightarrow \P^1_\Z$, up to automorphisms of $\P^1_\Z$. However,  \cite[Theorem~1]{ABS} states that for a positive proportion of cubic fields $K$, 
there does not exist an embedding of 
$\Spec\,\O_K\hookrightarrow \A^1_\Z$ even though there is no local obstruction to such an embedding.
 Similarly, Theorems~\ref{S4main} and \ref{thm:notbinary} of the current paper state that, for a positive proportion of {\it quartic} fields $K$, there does not exist a  closed immersion  $\Spec\, \O_K \hookrightarrow \A^1_\Z$
(resp.\ $\Spec\, \O_K \hookrightarrow \P^1_\Z$), despite there not being any local obstructions to the existence of such embeddings.
By  having no local obstruction, we mean that, for some  orientation $\omega \in \wedge^n_\Z \O_K$, there are  $\omega$-oriented closed immersions $\Spec \, \O_K \hookrightarrow\A^1_{\Z_p}$ (resp.\ $\Spec \, \O_K \hookrightarrow\P^1_{\Z_p}$) for all  $p$.  
  
\smallskip 
\subsection*{Methods}

While Theorems \ref{S4main} and \ref{thm:notbinary} are logically independent of each other, we prove them simultaneously by exhibiting a positive-density family of quartic fields that have no local obstruction to being monogenic (hence also have no local obstruction to being binary) and yet are not binary (hence are also not monogenic). Our method relies on connecting the  monogenicity of quartic rings to that of their cubic resolvent rings. Indeed, 
if a quartic ring is monogenic, then so is its cubic resolvent ring. More generally, if a quartic ring is binary, then its cubic resolvent ring is monogenic~\cite{Wood-monogenic}. 

It is thus natural to try and prove Theorems~\ref{S4main} and \ref{thm:notbinary} by considering the family~$\{K\}$ of non-monogenic cubic fields from~\cite{ABS}, and  for each such $K$ to construct a quartic field~$L$ such that the cubic resolvent ring of~$\O_L$ is~$\O_K$.  As such an $L$ would not be binary and have $\Disc(L) = \Disc(K)$, it would follow that a positive proportion of quartic fields are not binary (and hence not monogenic) when ordered by discriminant. There is even a natural candidate for $L$, namely the quartic field~$L_u$ (well defined up to conjugation) contained in the Galois closure of $K(\sqrt u)$, where $u \in \O_K^\times$ is any choice of non-square unit.   

There are a few obstacles in carrying out this plan. First, one cannot  guarantee that $\Disc(L_u) = \Disc(K)$ or, equivalently, that the cubic resolvent ring of $\O_{L_u}$ is the full ring of integers $\O_K$.  Indeed, given a cubic field  $K$, there might not be {any} quartic field $L$ with cubic resolvent field $K$ such that $\Disc(L)=\Disc(K)$. 
However, we have that  $\Disc(L_u)/\Disc(K)$ is uniformly bounded; we thus refine our proof in \cite{ABS} to prove that a positive proportion of the (possibly non-maximal) cubic resolvent rings of $\O_{L_u}$ are non-monogenic and yet have no local obstruction to being monogenic.

A more serious obstacle that arises is that we must prove that a positive proportion of the quartic fields $L_u$  themselves have no local obstruction to being monogenic. Indeed, it turns out that if $Q$ is a quartic ring with cubic resolvent ring $R$, then $Q$ can have a local obstruction to being monogenic  even if $R$ does not.  Thus, we must demonstrate the solubility of a positive proportion of the index form equations $f_{L_u}(x,y,z) = \pm 1$ over $\Z_p$, where $f_{L_u}$ is the ternary sextic index form corresponding to the quartic field $L_u$.  

We carry this out in two steps. We first observe that if the cubic resolvent of $Q$ is monogenic over $\Z_p$, then $\Spec \, Q \otimes_\Z \Z_p \simeq \mathrm{Proj}(\Z_p[x,y]/F(x,y))$ for some binary quartic form~$F$ over $\Z_p$. We then use a recent result of Fess (Theorem \ref{dan's identity}) to relate the  solubility of the index form equation to the solubility of the more manageable equation $z^2=F(x,y)$.  Imposing appropriate local conditions within our construction then guarantees that a positive proportion of the quartic fields $L_u$ have no local obstruction to being monogenic.  

After performing this local analysis, a suitable adaptation of the arguments in~\cite{ABS} then shows that a positive proportion of quartic fields are not  binary  despite having no local obstruction to being monogenic. Moreover, we exhibit a positive proportion of such quartic fields with positive (resp.\ negative) discriminant; in particular, the $r = 2$ cases of Theorems  \ref{S4main} and \ref{thm:notbinary} follow, since a quartic field has negative discriminant if and only if it has exactly two real embeddings. 

Using this approach to produce a positive proportion of quartic fields having each of the two remaining real signatures (cases $r = 0$ and $r = 4$ of Theorems \ref{S4main} and \ref{thm:notbinary}) turns out to be surprisingly subtle, as we must produce positive proportions of non-monogenic orders in cubic fields $K$ having totally positive non-square units $u$ as well as units $u$ of mixed signature.  On the one hand, we can guarantee that roughly $50\%$ of the cubic fields in the families constructed in \cite{ABS} are not monogenic. On the other hand, results of Bhargava-Varma \cite{BV2} give lower bounds on the proportion of cubic fields in such families having a unit $u$ of prescribed signature. For units of mixed signature,  which give rise to quartic fields of signature $(0,2)$, the lower bound of $75\%$ in \cite{BV2} is strong enough for us to conclude the existence of a positive proportion of mixed-signature quartic fields that are not monogenic.
However, for totally positive units, which give rise to totally real quartic fields, the lower bound of $50\%$ contained therein is not quite strong enough to guarantee overlap with our subset of non-monogenic cubic orders. To overcome this slight miss, we strengthen the method of \cite{ABS} just enough to force a bit of overlap between these two sets of density roughly $1/2$. Along the way we appeal to,  and prove a slight generalization of, results of Klagsbrun~\cite{Klagsbrun} on the average size of the $3$-torsion subgroups of  $S$-class groups of quadratic fields (for a fixed set $S$ of primes); the generalization allows the imposition of local conditions at finitely many primes not in $S$.

\subsection*{Contents}
This paper is organized as follows.  In \S\ref{sec:cubicprelims} and \S\ref{sec:quarticprelims}, we establish some facts about counting cubic and quartic fields, and we  discuss some of the key   properties of their index forms, including Fess's Theorem. In \S\ref{sec:construction}, we construct a positive-proportion family of quartic fields with no local obstruction to being monogenic. In \S\ref{sec:thm1proof}, we show that when we fix the sign of the discriminant, many of these quartic fields have non-monogenic cubic resolvent rings, thereby proving Theorems  \ref{S4main} and \ref{thm:notbinary} in the case $r = 2$. In \S\ref{sec:allsigs}, we exhibit a positive proportion of quartic fields (with no local obstructions to being monogenic) which are not binary and have signatures $(4,0)$ and $(0,2)$, respectively, thereby completing the proof of Theorems \ref{S4main} and \ref{thm:notbinary} for all real signatures. Finally, in \S\ref{aux}, we establish a mild generalization of Klagsbrun's results on the average size of 3-torsion in $S$-class groups that allows for specified local conditions at a finite set of primes not in $S$, which is utilized as an ingredient in  \S\ref{sec:allsigs}.  

\subsection*{Acknowledgments}
The authors thank Dan Fess and Arul Shankar for many helpful conversations.
The first author was supported by NSF grant~DMS-2002109. The second author was supported by a Simons Investigator Grant and NSF
grant~DMS-1001828. The third author was supported by the Israel Science Foundation (grant No.\ 2301/20).

\addtocontents{toc}{\protect\setcounter{tocdepth}{2}}

\section{Preliminaries on cubic rings}\label{sec:cubicprelims}
We recall the families of discriminants that were constructed in \cite{ABS}. Let $p_1,\ldots,p_t$ be primes satisfying $p_i \equiv2\pmod*3$, and let $n = 3\prod_{i=1}^t p_i$. Define
\[
\Sigma_n = \left\{\minus27dn^2 \in \Z \colon  d \mbox{ fundamental}, \, \left(\frac{d}{7}\right) \neq 1,  \mbox{ and } \left(\frac{d}{p}\right) = 1 \, \mbox{ for all } \, p \mid n \right\},
\]
and write $\Sigma_n^\pm$ for the subset of $D \in \Sigma_n$ such that $\pm D > 0$. 

Then we find that there are many cubic fields per discriminant in $\Sigma_n^\pm$, on average:  
\begin{proposition}\label{prop:averagecount}
Let $\Sigma_n^\pm(X) = \Sigma_n^\pm \cap [-X,X]$, and let $N(\Sigma_n^\pm,X)$ be the number of cubic fields with discriminant in $\Sigma_n^\pm(X)$. Then
\begin{align*}
\lim_{X \to \infty} \dfrac{N(\Sigma_n^\pm, X)}{\#\Sigma_n^\pm(X)} &=(2 \mp1)2^t, 
\end{align*}
\end{proposition}
\begin{proof}
This was stated without proof in \cite{ABS}, since the result was not used there.  We deduce the result from \cite[Thm.~8]{BST}.  
 To compute the local factors at $p \mid n$, we note that for $p \neq 3$, there is a unique totally ramified cubic extension of $\Q_p$, and it is not Galois --- and thus it has trivial automorphism group --- since $p \equiv 2\pmod*3$.  When $p = 3$, we note that there are three cubic extensions of $\Q_3$ of discriminant valuation 5, and they too have trivial automorphism groups. Therefore, we have
 \begin{align*}
 N(\Sigma_n^+, X) &\sim \frac{X}{12} \left(1 - 3^{-1}\right)\left(3\cdot 3^{-5}\right)\left(27\cdot 7^{-2}\right)\prod_{p\,  \nmid\, 7n}\left(1-p^{-1}\right)\left(1 + p^{-1}\right)\prod_{p\, \mid\, \frac n3} \left(1-p^{-1}\right)\left(p^{-2}\right),
 \\
 N(\Sigma_n^-, X) &\sim \frac{X}{4} \left(1 - 3^{-1}\right)\left(3\cdot 3^{-5}\right)\left(27\cdot 7^{-2}\right)\prod_{p \,\nmid\, 7n}\left(1-p^{-1}\right)\left(1 + p^{-1}\right)\prod_{p\,\mid\, \frac n3} \left(1-p^{-1}\right)\left(p^{-2}\right),
 \end{align*}
 while
 \begin{equation*} 
 \#\Sigma_n^\pm(X) \sim \frac{X}{2} \left(1 - 3^{-1}\right)\left(\textstyle\frac12\cdot 3^{-5}\right)\left(27\cdot 7^{-2}\right)\prod_{p \,\nmid\, 7n}\left(1-p^{-1}\right)\left(1 + p^{-1}\right)\prod_{p\,\mid\, \frac n3} \left(1-p^{-1}\right)\left(\textstyle\frac12\cdot p^{-2}\right).
 \end{equation*}
 the $\frac12$'s in the final Euler factor are due to the Kronecker symbol condition at each $p\mid n$.
\end{proof}

In fact, for many {\it individual} discriminants $D \in \Sigma_n$, we can prove that this average gives a lower bound on the number of cubic fields of discriminant $D$. Indeed, let $U_n^+$ (resp.\ $U_n^-$) be the subset of $D \in \Sigma_n^+$ (resp.\ $\Sigma_n^-$) such that the $3$-torsion in the class group of $\Q(\sqrt{-3D})$ (resp.\ $\Q(\sqrt{D})$) is trivial.  The following is stated as  \cite[Prop.\ 20--21]{ABS}.

\begin{proposition}\label{prop:r3=0}
If $D \in U_n^+$ $($resp.\ $U_n^-)$, then there are exactly $2^t$ $($resp.\ $3\cdot 2^t)$ cubic fields of discriminant $D$. 
\end{proposition}

\noindent 
We will also require a result of this type when $3$-part of the class group has rank 1.
\begin{proposition}\label{prop:modHCL}
Let $D \in \Sigma_n^+$, let $F =\Q(\sqrt{\minus3D})$, and factor $(p_i) = \p_i \bar\p_i$ in $\O_F$ for each $1\leq i\leq t$. If $\#\Cl(F)[3] = 3$ and at least one $[\p_i]$ is not a cube in $\Cl(F)$, then there are at least $3\cdot  2^{t-1}$ cubic fields $K$ having discriminant either $D$ or $D/p_i^2$.
\end{proposition}

\begin{proof}
Write $D = \minus27dn^2$ and fix some $f \mid n$. 
By the proof of \cite[Prop.\ 20]{ABS}, the number of cubic fields of discriminant $\minus27df^2$ is half the number of pairs $(J,[I])$, where $J\subseteq \O_F$ is an ideal of norm $f$ and $[I]\in \Cl(\O_F)$ a cube root of $[J]$. Note that if $[J]$ is a cube then, by hypothesis, it has exactly $3$ cube roots.

By assumption, there is an $i$ such that  $[\p_i]$ spans the one-dimensional $\F_3$-vector space $\Cl(F)/3\Cl(F)$.  For each of the $2^{t}$ ideals $J_1$ of norm $n/p_i$, there is a unique ideal $J_2$ of norm dividing $1$ or $p_i$ such that $[J_1] = [J_2]$ modulo cubes. The ideal $J := \bar J_1 J_2$ is therefore a cube in the class group, and has norm either $n$ or $n/p_i$.  Up to $\Gal(\O_F/\Z)$-conjugacy, this gives $ \frac123\cdot 2^{t} =  3\cdot 2^{t-1}$ pairs $(J, [I])$, and hence there are at least that many cubic fields $K$ of discriminant $D$ or $D/p_i^2$.
\end{proof}

\section{Preliminaries on quartic rings and their index forms}\label{sec:quarticprelims}
To study local obstructions to monogenicity for quartic fields, we use a parametrization of quartic rings and their cubic resolvent rings from \cite{HCLIII} and \cite{evan}, which in particular will allow us to relate the index form of a quartic ring to that of its cubic resolvent ring. 

For a principal ideal domain $S$, let $V(S)$ denote the space of pairs $(A,B)$ of ternary quadratic forms with coefficients in $S$. A {\it cubic} (resp.\ {\it quartic}) ring over $S$ is an $S$-algebra that is free of rank 3 (resp.\ 4) as an $S$-module.

\begin{theorem}[\cite{HCLIII},\cite{evan}]\label{thqrpar}
Let $S$ be a principal ideal domain.
There is a bijection between pairs $(A,B)\in V(S)$ and isomorphism classes of pairs $((Q,\langle \alpha,\beta,\gamma\rangle),(C,\langle\omega,\theta\rangle))$,
where~$Q$ is a quartic ring over $S$ with a chosen basis $\langle\alpha,\beta,\gamma\rangle$ of $Q/S$ and $C$ is a cubic resolvent
ring of~$Q$ with a chosen basis $\langle\omega,\theta\rangle$ of~$C/S$. 
\end{theorem}
\noindent 
A {\it cubic resolvent ring} $C$ of a quartic ring $Q$ over $S$ is a cubic ring over $S$ that is  equipped with an  isomorphism $\xi:\wedge^4 Q\to \wedge^3C$ and a quadratic mapping $\phi:Q/S\to C/S$ satisfying $\xi(1\wedge x\wedge y\wedge xy)=1\wedge\phi(x)\wedge\phi(y)$ for any $x,y\in Q/S$. These conditions also imply that  $\Disc(Q)=\Disc(C)$, when computed with respect to bases $\langle\alpha,\beta,\gamma\rangle$ and $\langle \omega,\theta\rangle$ of $Q$ and~$C$ such that $\xi(1\wedge\alpha\wedge\beta\wedge\gamma)=1\wedge\omega\wedge\theta$. By construction, the bases in Theorem~\ref{thqrpar} are compatible in this way.  See~\cite{HCLIII} and \cite{evan} for details.

The pair $(C,\langle\omega,\theta\rangle)$ is called a {\it based} cubic ring and 
$(Q,\langle\alpha,\beta,\gamma\rangle)$ a {\it based} quartic ring. 
A complete description of the construction of
$((Q,\langle\alpha,\beta,\gamma\rangle),(C,\langle\omega,\theta\rangle))$
from $(A,B)$ can be found
in~\cite[\S3.2--3.3]{HCLIII}. We denote the  (based) quartic ring and (based) cubic resolvent ring associated a pair $(A,B)\in V(S)$ by $Q(A,B)$ and $R(A,B)$, respectively. 
\noindent

We view elements of $(A,B)\in V(S)$ as pairs  of $3$-by-$3$ symmetric  matrices; we denote the associated quadratic forms by $q_A(x_1,x_2,x_3):=\sum_{1\leq i<j\leq 3}a_{ij}x_ix_j$ and $q_B(x_1,x_2,x_3):=\sum_{1\leq i<j\leq 3}b_{ij}x_ix_j$, respectively, where $a_{ij},b_{ij}\in S$. The index forms of $Q(A,B)$ and $R(A,B)$, in terms of $(A,B)$, are then given in the following proposition.

\begin{proposition}\label{prop: quarticcubic}
Let $(A,B)\in V(S)$. The index form $f_{R(A,B)}$ of $R(A,B)$ is given by \begin{equation}\label{cubicindex}
f_{R(A,B)}(x,y)= 4\cdot \det(A x - By),\end{equation} and the index form $f_{Q(A,B)}$ of $Q(A,B)$ is given by \begin{equation}\label{quarticindex}
f_{Q(A,B)}(x,y,z)= f_{R(A,B)}(q_B(x,y,z), \minus q_A(x,y,z)).
\end{equation}
In particular, if a quartic ring over $S$ is monogenic then its cubic resolvent ring is monogenic.
\end{proposition}

\begin{proof}
See \cite[Lemma $9$]{HCLIII} and the surrounding discussion for the assertions on the index forms.  A monogenic quartic ring $Q$ over $S$ is primitive and  thus, by~\cite[Corollary $4$]{HCLIII}, it has a unique cubic resolvent ring $R$.  Furthermore, if $f_Q$ represents 1, then by (\ref{quarticindex}), $f_R$ also represents~1! 
\qedhere
\end{proof}

There is a natural (based) ring of rank $n$ associated to any binary $n$-ic form over $S$~\cite{BM,Nakagawa,Wood-binary}. The quartic ring $Q_F$ over $S$ associated to a binary quartic form $F(x,y) =  a x^4 + b x^3 y + c x^2 y^2 + d x y^3 + e y^4$ over $S$ is given by  $Q(A_1, B_F)$, where
\begin{equation*}
A_1 = \left[\begin{array}{ccc} \phantom0 & \phantom0 & 1/2\\ \phantom0 & -1 & \phantom0\\ 1/2 & \phantom0 & \phantom0\end{array}\right],
\qquad B_F = \left[\begin{array}{ccc} a & b/2 & 0\\ b/2 & c & d/2\\0 & d/2 & e\end{array}\right];
\end{equation*}
see~Wood~\cite{Wood-monogenic,BS1}. 

\begin{theorem}[Fess~\cite{Fess}]\label{dan's identity}
Let $F$ be a binary quartic form over a principal ideal domain $S$, and  $Q_F$ the associated based quartic ring over $S$. Then the index form $f_{Q_F}$ of $Q_F$ satisfies
\begin{equation}
    f_{Q_F}(x^2, xy, y^2) = F(x,y)^3.
\end{equation}
\end{theorem}

\begin{proof}
We have $q_{A_1}(x^2, xy, y^2) = 0$ and $q_{B_F}(x^2, xy, y^2) = F(x,y)$. Thus, by (\ref{cubicindex}) and (\ref{quarticindex}), we see that $f_{Q_F}(x^2, xy, y^2) = 4\cdot \det(A_1\cdot F(x,y)) = F(x,y)^3$.
\end{proof}

\begin{corollary} 
If $F$ represents a square in $S^\times$,
then $f_{Q_F}(x,y,z)$ represents $1$ over $S$.
\end{corollary}
\begin{proof}
If $F(a,b) = c^2$ with $c\in S^\times$. Then  $c^{-1}\cdot (a^2, ab, b^2)$ is a solution to $f_{Q_F}(x,y,z) = 1$.
\end{proof}

\begin{proposition}\label{characterization of invariant orders of binary quartics}
Let $S = \Z$ or $\Z_p$ for some $p$. 
Suppose $((Q,\langle1,\alpha,\beta,\gamma\rangle),(R,\langle1,\omega,\theta\rangle))$ is a based quartic ring and a compatibly-based cubic resolvent ring over~$S$, such that the index form of $R$ with respect to this basis represents $1$.  
Then there is a binary quartic form $F$ over $S$ and compatible bases $\langle1,\alpha',\beta',\gamma'\rangle$ and $\langle1,\omega',\theta'\rangle$ for $Q$ and $R$ 
such that
\begin{equation*}
1\wedge\alpha\wedge\beta\wedge\gamma=1\wedge\alpha'\wedge\beta'\wedge\gamma',\quad 1\wedge\omega\wedge\theta=1\wedge\omega'\wedge\theta',
\end{equation*} and 
$((Q,\langle1,\alpha',\beta',\gamma'\rangle),(R,\langle1,\omega',\theta'\rangle))$ corresponds to $(A_1, B_F)$ via the correspondence of Theorem~$\ref{thqrpar}$.
In particular, the index forms $f_Q(x,y,z)$ and $f_R(x,y)$ associated to $(Q,\langle1,\alpha',\beta',\gamma'\rangle)$ and $(R,\langle1,\omega',\theta'\rangle)$, respectively, satisfy $$f_Q(x,y,z) = f_R(q_{B_F}(x,y,z), -q_{A_1}(x,y,z)).$$
\end{proposition}

\begin{proof}
Let $\langle1, \alpha, \beta, \gamma\rangle$ be an basis of $Q$ that is compatible via $\xi$ with the basis $\langle1, \theta, \theta^2\rangle$ of~$R$. Let $(A,B)\in V(S)$ be the corresponding pair under the bijection of Theorem \ref{thqrpar}. 
Because $f_R(1,0) = 1$, we see that $4\cdot \det(A) = 1$. Over $S$, there is a \emph{unique} $\GL_3(S)$-equivalence class of $A\in \Sym^2(S)^\ast$ for which $4\cdot \det(A) = 1$, namely the class of $A_1$.  Via a change of basis, we may assume $A = A_1$ and  $b_{13} = 0$.  The desired binary quartic form is then $F(x,y) := b_{11} x^4 + b_{12} x^3 y + b_{22} x^2 y^2 + b_{23} xy^3 + b_{33} y^4$, with $b_{ij}$ the coefficients of $B$.
\end{proof}

\begin{proposition}\label{condition on binary quartic to be unobstructed}
Let $p$ be a prime. Let $F\in \Z_p[x,y]$ be a binary quartic form and $Q_F$ the corresponding quartic ring over $\Z_p$. Let $\eps\in \{\pm 1\}$. If the curve $\eps z^2 = F(x,y)$ over $\F_p$ has a smooth $\F_p$-point with $z\neq 0$, then the index form of $Q_F$ represents $\eps$.
\end{proposition}

\begin{proof}
Via Hensel lifting, we have $x,y\in \Z_p$ and $z\in \Z_p^\times$ such that $\eps z^2 = F(x,y)$. Then, by Theorem \ref{dan's identity}, we have
$I_{Q_F}(z^{-1}\cdot (x^2, xy, y^2)) = z^{-6} F(x,y)^3 = \eps^3 = \eps.$
\end{proof}

\begin{lemma}\label{lem:reps1}
Let $F\in \F_p[x,y]$ be a quartic form with splitting type $(1^3 1)$ or $(1^2 2)$.  Then the curve $z^2 = F(x,y)$ has an $\F_p$-point with $z\neq 0$. If $p > 3$, then the same holds for $F$ of type $(1^2 1 1)$ as well.
\end{lemma}

\begin{proof}
We first consider the case $(1^3 1)$.  More generally, if $L_1$ and $L_2$ are linearly independent linear forms and $n\geq 1$, then the equation $1 = L_1(v)^n L_2(v)$ has a solution $v\in \F_p^2$. Indeed, the map $(L_1(v), L_2(v)): \F_p^2 \to \F_p^{2}$ is surjective by hypothesis, so $(1,1)$ is in its image.

If $F$ has type $(1^2 1 1)$ or $(1^2 2)$, then  $F = L_0^2 Q$ for some linear form $L_0$ and quadratic form $Q$ with nonzero discriminant. The conic $z^2 = Q(x,y)$ has $p+1 > 2$ points over~$\F_p$. If $Q$ is irreducible, then none of these points have $z = 0$, and at most two of them intersect the line $L_0 = 0$.  Hence  if $F$ has splitting type $(1^22)$, then $z^2 = F(x,y)$ has an $\F_p$-point with $z \neq 0$.
If $Q = L_1L_2$ is reducible, then the conic intersects $L_0L_1L_2 = 0$ in at most 4 points. Thus, if $p > 3$, then $p + 1 > 4$ and there is at least one solution to $z^2 = F(x,y)$ with $z \neq 0$.
\end{proof}

Finally, we show that the index form of an order in a number field of any degree always represents $1$ and $-1$ over $\R$, and hence we need not consider archimedean places when studying local obstructions to monogenicity.

\begin{lemma}\label{lem:infinity}
The index form $f_\O$ of an order $\O$ in a number field $L\neq \Q$ represents both $1$ and $-1$ over $\R$.
\end{lemma}
\begin{proof}
Let $n\geq 2$ be the degree of $L$, and let $\sigma_1,\ldots,\sigma_n$ denote the distinct  embeddings $L\to\C$.
Let $ 1,\theta_1,\ldots,\theta_{n-1}$ be a $\Z$-basis for $\O$. Then, over $\C$, we have the factorization
\begin{equation*}
f_\O(x_1,\ldots,x_{n-1}) = \pm   
\frac1{\sqrt{\Disc(\O)}}\prod_{1\leq i<j\leq n}
\left(\sum_{k=1}^{n-1}
(\sigma_i(\theta_k)-\sigma_j(\theta_k))x_k\right).
\end{equation*}\noindent
It suffices to check that at least one of the linear factors of $f_\O$ above is in fact defined over~$\R$. Indeed, if $L$ is totally real, then all factors are defined over~$\R$. Otherwise, choose a factor corresponding to $(i,j)$ where $\sigma_i$ and $\sigma_j$ are complex conjugate embeddings; then, upon removing a factor of $\sqrt{-1}$, we see that this factor is defined over $\R$. 
Since $f_{\O_L}$ has a real linear factor, it takes on all real values, and in particular both  $1$ and $-1$.\end{proof}

\section{Construction of quartic fields with no local obstruction to being monogenic}\label{sec:construction}
We use the family 
of discriminants $\Sigma_n$ from Section \ref{sec:cubicprelims} and assume that $n$ is a multiple of $30=2 \cdot 3 \cdot 5$.
We denote the family of cubic fields  having discriminants in $\Sigma_n$ by $F_3$. 

For each $K\in F_3$, let $\mathcal{U}_K$ be a set of representatives for the nonzero elements of the finite group $\left(\O_K^\times / \O_K^{\times2}\right)_{\Nm = 1}$ of norm-one square-classes in $\O_K^\times$. This set has size 3 if $\Disc(K) > 0$  and size 1 if $\Disc(K) < 0$. For each $u\in \mathcal{U}_K$, since $K$ is an $S_3$-cubic field, the Galois closure of $K(\sqrt{u})$ is an $S_4$-field \cite[Lem.\ 4--5]{baily}, and thus contains a quartic field $L_u$, unique up to conjugacy. Consider the family $F_4 = \{L_u : K\in F_3, u\in\mathcal{U}_K\}$ consisting of all of the quartic fields $L_u$ arising in this way from all $K\in F_3$.

\begin{lemma}\label{lem:index2}
Let $L\in F_4$, and let $K$ be its cubic resolvent field. Let $R$ be the cubic resolvent ring of $\O_L$. Then $R \otimes_\Z \Z_p \simeq \O_K \otimes_\Z \Z_p$ for $p > 2$.  If $p = 2$, then $R\otimes_\Z \Z_2$ is isomorphic to $R_0\simeq \Z_2[\sqrt[3]{2}]$, $R_1\simeq \Z_2[\sqrt[3]{4}]$, or $R_3\simeq \Z_2[2(\sqrt[3]{2}+\sqrt[3]{4})]$, and hence the index of $R$ in $\O_K$ is equal to $2^r$ for  $r = 0$, $1$, or $3$, respectively. 
\end{lemma}
\begin{proof}
Since the extension $K(\sqrt{u})/K$ is unramified away from $2$, the discriminant $\Disc(L)$ agrees with $\Disc(K)$ up to powers of $2$. By the definition of $\Sigma_n$, the cubic field $K$ is totally ramified at $2$, and since the unique totally ramified cubic extension of $\Q_2$ is $K_2 :=\Q_2(\sqrt[3]{2})$, we have $K \otimes_\Q \Q_2 \simeq K_2$ and $\O_{K_2}\simeq R_0$.  The group of squareclasses of norm-one units in $R_0^\times$ is of order 4. The trivial squareclass gives rise to the quartic $\Q_2$-algebra $L_0=K_2 \times \Q_2$ having discriminant valuation $2$, while the three nontrivial elements give rise to three distinct $S_4$-quartic fields $L_1$, $L_3$, and $L_3'$ over $\Q_2$ with discriminant valuations $4$, $8$, and $8$, respectively. A computation shows the cubic resolvent rings of $L_0$, $L_1$, $L_3$, and $L_3'$ to be $R_0$, $R_1$, $R_3$, and~$R_3$, respectively, which have the stated indices in $R_0$.
\end{proof}

\begin{lemma}\label{lem:resolventring}
Let $L\in F_4$, and let $R$ be the cubic resolvent ring of $\O_L$.  Then $R$ is has no local obstruction  to being monogenic.
\end{lemma}

\begin{proof}
By Lemma \ref{lem:index2}, $R \otimes_\Z \Z_2$ is monogenic, and hence the index form $f_R$ is a primitive binary cubic form. Since $\Disc(f_R)$ is divisible by $2$, \cite[Lem.\ 9]{ABS} implies that $f_R$ represents both $+1$ and $\minus 1$ over~$\Z_2$. For $p > 2$, the ring $R \otimes_\Z \Z_p$ is maximal, and the proof is exactly as in \cite[Lem.\ 18]{ABS}.   
\end{proof}

\begin{proposition}\label{unob'}
If $L\in F_4$, then $L$ has no local obstruction to being monogenic.
\end{proposition}
\begin{proof}
Let $R$ be the cubic resolvent ring of $\O_L$, and set $K = R \otimes_\Z \Q$. 

First assume $p > 2$. Since $R$ is locally unobstructed (Lemma \ref{lem:resolventring}), we have $\O_L \otimes_\Z \Z_p \simeq Q_F$ for some binary quartic form $F \in \Z_p[x,y]$, and $f_{\O_L}(x,y,z) = f_{Q_{F}}((x,y,z)\cdot g)$ for some $g\in \SL_3(\Z_p)$ by Proposition~\ref{characterization of invariant orders of binary quartics}. We will show that both 
$f_{Q_{F}}(x,y,z) 
=1$ and 
$f_{Q_{F}}(x,y,z) 
=-1$ 
have $\Z_p$-solutions. By Proposition \ref{condition on binary quartic to be unobstructed}, it suffices to show that each curve $C_{F, \eps}: \eps z^2 = F(x,y)$ with $\epsilon\in\{\pm1\}$ has an $\F_p$-point with $z\neq 0$.  Because $Q_{F}\simeq \O_L\otimes_\Z \Z_p$ is maximal, the splitting type of $F$ modulo $p$ matches the splitting type of $p$ in $L$, and we break into cases based on whether $p$ divides $\Disc(K) = \minus27dn^2$ or not.

If $p\nmid dn$, then $ p \nmid \Disc(F)$ and, as long as $p > 7$, the Hasse--Weil bound implies that $C_{F, \eps}$ has an $\F_p$-point with $z\neq 0$. For $p\mid d$, the splitting type of $p$ in $K$ is $(1^2 1)$, and so in $L$ it is $(1^2 2)$ or $(1^2 1 1)$. Thus $C_{F, \eps}$ has an $\F_p$-point with $z\neq 0$ thanks to Lemma \ref{lem:reps1}; note that $p > 3$ when $p \mid d$, since we have insisted that $6 \mid n$.  If $p = 7$ and $p\nmid d$, the splitting type of $p$ in $K$ is $(12)$ because $\left(\frac{d}{7}\right)\neq 1$. Hence the splitting type of $p$ in $L$ is either $(4)$ or $(112)$, and so $F$ has at most $2$ roots over $\F_p$. Thus the Hasse--Weil bound implies that $C_{F, \eps}$ has an $\F_7$-point with $z\neq 0$.   For $p\mid n$ (which includes $p = 3$ and $p = 5$, by assumption), the splitting type of $p$ in $K$ is $(1^3)$, so that the splitting type of $p$ in $L$ is $(1^3 1)$. Hence  the splitting type of $F$ modulo $p$ is $(1^3 1)$ and Lemma \ref{lem:reps1} implies $C_{F, \eps}$ has an $\F_p$-point with $z\neq 0$.

Next assume $p = 2$. We prove that 
the index form of $\O_L$ represents 
$+1$ or $\minus 1$ over~$\Z_2$. Since $2$ is totally ramified in $K$, we have 
$K\otimes_\Q \Q_2 \simeq \Q_2(\sqrt[3]{2})$.  Hence either $L\otimes_\Q \Q_2\simeq \Q_2(\sqrt[3]{2}) \times \Q_2$, implying that the splitting type of $2$ in $L$ is $(1^3 1)$, or else $L\otimes_\Q \Q_2$ is a field in which $2$ splits as $(1^4)$.  

In the first case, we conclude by the same argument that we used to deal with the primes $p\mid n$ with $p > 2$.
In the second case, $L\otimes_\Q \Q_2$ is an $S_4$-quartic field. 
As mentioned earlier, there are three quartic $S_4$-extensions of $\Q_2$. 
Explicitly, we have $\O_L\otimes_\Z \Z_2\cong \Z_2[x]/(g)$ for some $g\in \{x^4 + 2x + 2, x^4 + 4x + 2, x^4 + 4x^2 + 4x + 2\}$. 
Therefore, 
there exists $u\in \Z_2^\times$ such that $f_{\O_L}(x,y,z) = u \cdot f_{Q_g}(x,y,z)$.  Thus it suffices to check that, for all three quartic forms $g$
and all units $u\in \Z_2^\times$, the equation 
$f_{Q_g}(x,y,z) = 1$ or $f_{Q_g}(x,y,z) = \minus 1$ 
has a solution. Since we may scale $u$ by an element of $(\Z_2^\times)^6 = (\Z_2^\times)^2 = 1 + 8\Z_2$, we need only verify this statement modulo $8$, which is a quick check.
\end{proof}

\section{Proof of Theorems \ref{S4main} and \ref{thm:notbinary}: the case $r = 2$}\label{sec:thm1proof}
We prove that a positive proportion of positive discriminant (resp.\ negative discriminant) quartic fields are not even binary, despite having no local obstruction to being monogenic. The result for negative discriminants immediately implies the case $r \!=\! 2$ of Theorems \ref{S4main} and~\ref{thm:notbinary}. 

For each choice of sign $\pm$, write $F_3^\pm$ and $F_4^\pm$ for the subsets of $F_3$ and $F_4$ having  discriminants of sign $\pm$. 
Recall from Section \ref{sec:cubicprelims} the subset $U_n^\pm \subset \Sigma_n^\pm$. 
By Proposition \ref{prop:r3=0}, if $D\in U_n^\pm$, then the number of $L\in F_4$ 
having cubic resolvent field $K$ of discriminant $D$
is 
\[\bigl[\#\bigl(\O_K^\times / \O_K^{\times2}\bigr)_{\Nm = 1} - 1\bigr]  \cdot (2 \mp 1)  2^t = 3\cdot 2^t.\] 
For $D\in \Sigma_n$, let $B_D$ denote the number of binary fields $L\in F_4$ having cubic resolvent field of discriminant $D$.

For each nonzero $D \in \Z$, we define the elliptic curve $E^D \colon y^2 = 4x^3 + D$. If $R_L$ is the cubic resolvent ring of $\O_L$, then $f_{R_L}$ gives rise to two nonzero elements $\alpha_L\in\Sel_3(E^{\Disc(L)})$ by Lemma \ref{lem:resolventring} and the construction of \cite[\S3.2]{ABS}.  Moreover, $\alpha_L = \pm \alpha_{L'}$ if and only if $R_L \cong R_{L'}$, as in the proof of \cite[Lem.\ 15]{ABS}.  If $L$ is binary, then $R_L$ is monogenic by Proposition \ref{characterization of invariant orders of binary quartics}, and hence $\pm \alpha_L$ both lie in  $E^{\Disc(L)}(\Q)/3E^{\Disc(L)}(\Q)$.
Thus, by Lemma \ref{lem:index2},
$$9 + 2B_D\leq 3(3^{\rank{E^D}} + 3^{\rank{E^{4D}}} + 3^{\rank{E^{64D}}}).$$ Since  $g(x):=x^{\log_3{2}}$ is a subadditive function, i.e., $g\!\left(\sum a_i\right)\leq \sum g(a_i)$, we have
$$(9 + 2B_D)^{\log_3{2}}\leq 2(2^{\rank{E^D}} + 2^{\rank{E^{4D}}} + 2^{\rank{E^{64D}}}).$$ 
Denote by $\mu_{n,\pm}$ the lower density of $U_n^\pm$ in $\Sigma_n^\pm$, which is at least $1/2$ by \cite{BV}. We have
\begin{equation*}
 \mu_{n,\pm}\cdot \avg_{D\in U_n^\pm}(9 + 2 B_D)^{\log_3{2}}
+ (1 - \mu_{n,\pm})\cdot 3
 \:\leq\: 2(3+3+3)\:=\:18, 
\end{equation*}
by \cite[Thm.\ 13]{ABS}.
Therefore, $$\avg_{D\in U_n^\pm} (9 + 2 B_D)^{\log_3{2}}\leq 3 + 15 \mu_{n,\pm}^{-1}\leq 33.$$
Given the constraints (a) $B_D\in [0,3\cdot 2^t]$ for all $D \in U_n^\pm$, and (b) $\avg_{D\in U_n^\pm}(9+2B_D)^{\log_32}\leq 33$, the maximum of $\avg_{D\in U_n^\pm}(9+2B_D)$ is achieved, for $t \geq 4$, when $B_D$ attains its maximum value $3 \cdot 2^t$ for a density of $24/((9+3\cdot 2^{t+1})^{\log_32}-1 )$ of $D\in U_n^\pm$,  and $B_D=0$  otherwise.  It follows that a positive lower density of at least $1-O(2^{-t\log_32})$ quartic fields~$L$ with cubic resolvent field of discriminant $D\in U_n^\pm$ are non-monogenic. 

Thus if $t \geq 5$, 
the number of $L\in F_4$ which are not binary and are of absolute discriminant at most~$X$ is $\gg X$. By Proposition \ref{unob'}, these quartic fields are all locally unobstructed to being monogenic.
 By~\cite{BhargavaQuartics}, they give a positive proportion of all quartic fields when ordered by absolute discriminant, completing the proof.  \hfill {\;$\;\Box$\!\! \vspace{2 ex}}

\section{Proof of Theorems \ref{S4main} and \ref{thm:notbinary}: the cases $r= 0$ and $r = 4$}\label{sec:allsigs}

We  first treat  the case $r=0$.  
We use a result of the second author and Varma \cite[Cor.~10]{BV2}, which states that at least $\frac34$ of cubic fields in $F^+_3$ have a unit of mixed signature.  By Proposition \ref{prop:averagecount}, the proportion of cubic fields in $F_3^+$ with discriminant in $U_n^+$ is $\mu_n^+ \geq 1/2$.  Hence, at least $\frac12$ of cubic fields with $\Disc(K) \in U_n^+$ have a unit of mixed signature. For such cubic fields, two of the three corresponding quartic fields in $F_4^+$ are totally complex.  On the other hand, for $t\gg 1$, the proof of Theorem \ref{S4main} shows that a proportion of more than $1 - O(2^{-t \log_3{2}})$ of quartic fields in $F_4^+$ with cubic resolvent discriminant in $U_n^+$ are not binary. It follows that at least $\frac23\cdot \left(1 - \frac{1}{2} - O(2^{-t \log_3{2}})\right) = \frac{1}{3} - O(2^{-t \log_3{2}})$ of the quartic fields produced in the proof of Section \ref{sec:thm1proof} are totally complex, proving Theorem~\ref{S4main} in the case $r = 0$. (In fact, $t = 7$ suffices to make the latter density positive.)

Finally, we turn to the totally real case $r=4$.  
Let $\widetilde{\Sigma}_n^+$ be the set of integers of the form $D$ or $D/p_i^2$, for some $D \in \Sigma_n^+$ and some $1 \leq i \leq t$. 
Let $U_{n,1}^+$ be the subset of $D\in \Sigma_n^+$ such that $\#\Cl(\Q(\sqrt{\minus3D}))[3] = 3$ and such that not all of the $[\p_i]$ are cubes in $\Cl(\Q(\sqrt{\minus3D}))$. Let $\mu_{n,1}^+$
be the lower relative density of $U_{n,1}^+$
in $\Sigma_n^+$.
Write also $U_{n, \leq 1}^+ := U_n^+ \cup U_{n,1}^+$ and $\mu_{n,\leq 1}^+ := \mu_n^+ + \mu_{n,1}^+$.

\begin{lemma}\label{lem:classgrpprops}
We have 
$\mu_{n,\leq 1}^+\geq \frac{7}{8} - \frac{1}{2}\cdot 3^{-t}$.
\end{lemma}

\begin{proof}
Let $\widetilde{U}_{n,1}^+\supset U_{n,1}^+$ be the subset of $D\in \Sigma_n^+$ such that $\#\Cl(\Q(\sqrt{\minus3D}))[3] = 3$, and let $\tilde{\mu}_{n,1}$ be the lower relative density of $\widetilde{U}_{n,1}^+$ in $\Sigma_n^+$. Then $\mu_{n,1}^+\leq \tilde{\mu}_{n,1}^+\leq \mu_{n,1}^+ + \frac{1}{2}\cdot 3^{-t}$ by Corollary~\ref{cor:klag} in \S\ref{aux}.  By \cite[Cor.\ $4$]{BV}, we have
$\avg_{D\in\Sigma_n^+}\#\Cl(\Q(\sqrt{\minus3D}))[3] = 2$, and 
hence
$$(\mu_n^+ + \tilde{\mu}_{n, 1}^+)\cdot 1 + (1 - \mu_n^+ - \tilde{\mu}_{n,1}^+)\cdot 9\leq 2.$$
We conclude that 
$\mu_n^+ + \tilde{\mu}_{n, 1}^+\geq \frac78$, and therefore 
$\mu_n^+ + \mu_{n,1}^+\geq \frac{7}{8} - \frac{1}{2}\cdot 3^{-t}$.
\end{proof}

By Proposition \ref{prop:averagecount}, the total number of cubic fields with discriminant in $\widetilde{\Sigma}_n^+$ and  corresponding $D \in \Sigma_n^+$ such that $D < X$, is 
\begin{equation}\label{eq:1}
    \sim \left(2^t+ t \cdot 2^{t-1}\right) X = 2^{t-1}(t+2)X.
\end{equation}
On the other hand, the total number of cubic fields with discriminant in $\widetilde{\Sigma}_n^+$, with corresponding $D \in U_{n, \leq 1}^+$ such that $D < X$, is at least
\begin{align}\label{eq:2}
\sim \left(\mu_n^+(t \cdot 2^{t-1} + 2^t)  + \mu_{n,1}^+(3 \cdot  2^{t-1})\right)X
&\geq 
\left(\mu_n^+ (t-1)2^{t-1}
+ (\mu_n^+ + \mu_{n,1}^+)(
3\cdot 2^{t-1})\right)X \nonumber\\
&\geq \left(\textstyle\frac12 (t-1)2^{t-1}+(\textstyle\frac78-\textstyle\frac12\cdot3^{-t})(3\cdot 2^{t-1})\right)X
\end{align}
by Propositions \ref{prop:r3=0}-\ref{prop:modHCL} and Lemma~\ref{lem:classgrpprops}. 
The lower density of (\ref{eq:2}) in (\ref{eq:1}) is thus at least
\begin{equation*}
\frac{\textstyle\frac12 (t-1)2^{t-1}+(\textstyle\frac78-\textstyle\frac12\cdot3^{-t})(3\cdot 2^{t-1})}
{2^{t-1}(t+2)}
 = \frac12 + \frac{9-4\cdot 3^{-t+1}}{8(t+2)}.
\end{equation*}
as $X\to\infty$. 

We may now conclude as in Section \ref{sec:thm1proof}. For $D \in \Sigma_n^+$, let $B_D$ be the number of monogenic quartic fields with cubic resolvent field of discriminant $D$ or $D/p_i^2$ for some $i$.  Then 
$$9(t+1) + 2B_D\leq 3\biggl(3^{\rank{E^D}} + 3^{\rank{E^{4D}}} + 3^{\rank{E^{64D}}}+\sum_{i=1}^t
\Bigl(3^{\rank{E^{D/p_i^2}}} + 3^{\rank{E^{4D/p_i^2}}} + 3^{\rank{E^{64D/p_i^2}}}\Bigr)\biggr)$$ and so
$$(9(t+1) + 2B_D)^{\log_32}\leq 2\biggl(2^{\rank{E^D}} + 2^{\rank{E^{4D}}} + 2^{\rank{E^{64D}}}+\sum_{i=1}^t
\Bigl(2^{\rank{E^{D/p_i^2}}} + 2^{\rank{E^{4D/p_i^2}}} + 2^{\rank{E^{64D/p_i^2}}}\Bigr)\biggr).$$
By Theorem \cite[Thm.\ 13]{ABS}, we have
\[
 \mu_{n,\leq 1}^+\cdot \avg_{D\in U_{n,\leq 1}^+}(9(t+1) + 2 B_D)^{\log_3{2}}
+ (1 - \mu^+_{n,\leq1})\cdot (3 (t+1))
 \:\leq\: 18 (t + 1).
\]
Hence, using Lemma \ref{lem:classgrpprops}, we have
\[\avg_{D\in U_{n,\leq 1}^+}(9(t+1) + 2 B_D)^{\log_3{2}} \:\leq\: (t+1)(15\cdot (\mu_{n,\leq 1}^+)^{-1} + 3) \:\leq\: 21(t + 1).\]
For $D \in U_{n,\leq 1}^+$, the maximum possible value of $B_D$ is $3\cdot 2^t + t\cdot 3\cdot 2^{t-1} = 3\cdot 2^{t-1}(t +2)$. Since (a) $B_D\in [0,3\cdot 2^{t-1}(t + 2)]$ for all $D \in U_{n,\leq 1}$, and (b) 
$\avg_{D\in U_{n,\leq 1}}(9(t+1)+2B_D)^{\log_32}\leq 21(t+1),$ the maximum of $\avg_{D\in U^+_{n,\leq1}}(9(t+1)+2B_D)$ is achieved when $B_D$ takes its maximum value $3 \cdot 2^{t-1}(t + 2)$ for a density of $21(t+1) (3(t+2) 2^{t-1})^{-\log_32}$ of $D\in U_{n, \leq 1}$,  and $B_D=0$  otherwise.

We conclude that the total number of binary quartic fields with cubic resolvent field of discriminant $D$ or $D/p_i^2$ for $D\in U_{n,\leq 1}^+$ and $D < X$ is $O(t^{2 - \log_3{2}}\cdot 2^{t\cdot (1 - \log_3{2})}X)$. Since, for each $D\in U_n^+\subset U_{n,\leq 1}^+$, there are at least $2^{t-1}(t+2)$ quartic fields with cubic resolvent field of discriminant $D$ or $D/p_i^2$ by Proposition \ref{prop:r3=0}, it follows that, ordering by the corresponding $D\in U_{n,\leq 1}^+$, the probability that such a quartic field 
is binary is $O(t^{1-\log_3{2}}\cdot 2^{-t\log_3{2}})$.

By \cite[Cor.\ 9]{BV2}, at least $\frac12$ of all cubic fields with discriminant in $\widetilde{\Sigma}_n^+$ have a totally positive non-square unit.  Hence a proportion of at least \[\frac12 + \frac{9-4\cdot 3^{-t+1}}{8(t+2)} - \frac12 
 \:=\: \frac{9-4\cdot 3^{-t+1}}{8(t+2)}\]
of the cubic fields with discriminant in $U_{n,\leq1}^+$ have a totally positive non-square unit, and hence give rise to a totally real quartic field.  The proportion of these which are not binary is at least $\frac{9}{8(t+3)} - O(t^{1 - \log_3{2}}\cdot 2^{-t \log_3{2}})$, which is positive for $t \gg 1$. 

\section{An auxiliary result on the average number of $3$-torsion elements in $S$-class groups}\label{aux}
We state a slight generalization of a result of Klagsbrun~\cite{Klagsbrun} on the average number of $3$-torsion elements in the $S$-class groups of quadratic fields, that allows specified local conditions at a  finite set of primes not in $S$, and which was used in the proof of Theorem~\ref{S4main}. 

For each
prime $p$, let $\Sigma_p$ be a set of isomorphism classes of  \'etale quadratic algebras over~$\Q_p$. The collection $(\Sigma_p)_p$ is {\it acceptable}  if, for all but finitely many $p$, the set
$\Sigma_p$ contains all \'etale quadratic algebras over $\Q_p$. Let $\Sigma$ denote the set of quadratic fields $F$,
up to isomorphism, such that $F\otimes_\Q \Q_p\in\Sigma_p$ for each $p$. 

\begin{theorem}\label{zevgen}
Let $S$ be a finite set of primes, and let $\Sigma=(\Sigma_p)_p$ be any acceptable set of quadratic fields defined by local conditions such that, for $p\in S$, the set  $\Sigma_p:=\{\Q_p\times\Q_p\}$ consists of just the unique split \'etale quadratic extension of $\Q_p$. Then, when quadratic fields with discriminants in $\Sigma$ are ordered by absolute discriminant:
\begin{enumerate}
\item[$(i)$] the average size of $\Cl(F)_S[3]$, as $F$ ranges over imaginary quadratic fields in $\Sigma$, is equal to $1 + 3^{-|S|};$  and
\item[$(ii)$] the average size of $\Cl(F)_S[3]$, as $F$ ranges over real quadratic fields in $\Sigma$, is equal to $1 + 3^{-|S|-1}$.
\end{enumerate}
\end{theorem}
\begin{proof}The proof proceeds exactly as in \cite{Klagsbrun}, but using the proof method of \cite{BV} to impose additional local conditions at primes not in $S$. All Euler  factors at primes $p\in S$ are identical to those in the proof of  \cite[Theorem~2]{Klagsbrun}, while all Euler factors at primes $p\notin S$ also remain unchanged due to the  identical cancellation as explained in \cite[Proof of Corollary~4, \S5.4]{BV}.  Thus the averages in \cite[Theorem~2]{Klagsbrun} remain the same  even upon the imposition of local conditions at additional primes not in $S$. 
\end{proof}

\begin{corollary}\label{cor:klag}
Write $F = \Q(\sqrt{\minus3D})$.  The proportion of $D \in \Sigma_n^\pm$ such that for each $1\leq i\leq t$, the primes $[\p_i]$ are cubes in $\Cl(F)$, and such that $\#\Cl(F)[3] = 3$, is at most $\frac{1}{2}\cdot 3^{-t}$.
\end{corollary}

\begin{proof}
Let $S := \{p_1, \ldots, p_t\}$. 
Since $|\Cl(F)_S[3]| = |\Cl(F)[3]|=3$ and 
$\avg_{D\in \Sigma_n^{\pm}}|\Cl(\Q(\sqrt{D})_S[3]|$ 
is at most $1 + 3^{-t}$ by Theorem \ref{zevgen}, it follows that
\begin{equation*}
\Prob_{D\in \Sigma_n^\pm}(\#\Cl(F)_S[3] = \#\Cl(F)[3] = 3)\:\leq\: \avg_{D\in \Sigma_n^\pm}\, \frac{|\Cl(F)_S[3]| - 1}{2}\:\leq\: \frac{1}{2}\cdot 3^{-t}.
\end{equation*}
\end{proof}

\bibliographystyle{abbrv}
\bibliography{references}

\end{document}